\documentclass[reqno]{amsart}
\usepackage{amsmath, amsthm, amscd, amssymb, amsfonts, amsbsy}
\usepackage{mathrsfs}
\usepackage{pxfonts}
\usepackage{bbm}
\usepackage{todonotes}
\usepackage{marginnote}

\theoremstyle{plain}
\newtheorem{theorem}[equation]{Theorem}
\newtheorem{lemma}[equation]{Lemma}

\theoremstyle{definition}

\theoremstyle{remark}

\numberwithin{equation}{section}

\providecommand{\abs}[1]{\lvert#1\rvert}
\providecommand{\Abs}[1]{\left\lvert#1\right\rvert}

\providecommand{\norm}[1]{\lVert#1\rVert}

\renewcommand{\vec}[1]{\boldsymbol{#1}}

\begin{document}
\title[Kelly's conjectures for odd electorates]
{Kelly's conjectures on Condorcet winner probabilities for odd electorates}

\author[S. Kim]{Seick Kim}
\address[S. Kim]{Department of Mathematics, Yonsei University, 50 Yonsei-Ro, Seodaemun-gu, Seoul 03722, Republic of Korea}
\email{kimseick@yonsei.ac.kr}
\thanks{S. Kim is supported by the National Research Foundation of Korea (NRF) under agreement NRF-2022R1A2C1003322.
S. Kim is affiliated with the Intelligent Computational Science Institute (IN2CSI), Yonsei University, 50 Yonsei-Ro, Seodaemun-gu, Seoul 03722, Republic of Korea.}

\author[H. Shin]{Hocheol Shin}
\address[H. Shin]{Interdisciplinary Program in Philosophy, Politics and Economics, Seoul National University,  1 Gwanak-ro, Gwanak-gu, Seoul 08826, Republic of Korea}
\email{hshin@snu.ac.kr}

\subjclass[2020]{Primary 91B12; Secondary 60C05, 91B14}

\keywords{Condorcet winner, Kelly's conjectures, impartial culture, odd electorates, strict monotonicity, Bernoulli sums}

\begin{abstract}
Let $C(m,n)$ denote the probability that, among $n$ independent uniformly random permutations of $m$ labels, there exists a label that precedes every other label in a strict majority of the permutations.
We prove that $C(m,n)>C(m,n+2)$ for $m \ge 3$ and odd $n \ge 1$, and that $C(m,n)>C(m+1,n)$ for $m \ge 2$ and odd $n \ge 3$.
These results establish Kelly's two monotonicity conjectures for odd electorates.
The proofs use integral representations involving probabilities associated with sums of independent Bernoulli random variables.
\end{abstract}

\maketitle

\section{Introduction and main results}

Let $n$ voters independently choose uniformly random permutations of the labels $\{1,\ldots,m\}$.
This is the impartial culture model, in which each of the $m!$ rankings is equally likely.
A label $a\in\{1,\ldots,m\}$ is called a majority winner (or Condorcet winner) if, for every $b\ne a$, more than half of the permutations place $a$ before $b$.
There is at most one majority winner.
Let $C(m,n)$ denote the probability that one exists.

A majority winner need not exist.
For example, consider three labels $a, b, c$ and three voters with rankings $abc$, $bca$, and $cab$, written from most to least preferred.
These give a two-to-one majority for $a$ over $b$, $b$ over $c$, and $c$ over $a$.
Thus, there is no majority winner.
This is the classical Condorcet paradox.

Black \cite{Black1948} suggested that, for a fixed number of voters, majority winners become less frequent
as the number of alternatives increases.
Kelly \cite{Kelly1974, Kelly1986} studied how the probability of a Condorcet winner depends on the numbers of voters and
alternatives and formulated two monotonicity conjectures.
His definition allows ties in pairwise comparisons, but coincides with ours when $n$ is odd.
Fishburn, Gehrlein, and Maskin \cite{FGM1979} subsequently investigated these conjectures and established several special cases.
For further results on Condorcet winner probabilities, including exact calculations, partial results, and
asymptotic estimates, see \cite{DP1970, May1971, Fishburn1973, GF1976, KR2005, Gehrlein2006, Sauermann2022, CDNPR2025, Pittel2025} and the references therein.

In this paper, we prove both conjectures for odd numbers of voters.
Specifically, we show that the probability of a Condorcet winner strictly decreases upon adding either two voters or one alternative.

Note that $C(2,n)=1$ for every odd $n$, since a majority contest between two labels cannot end in a tie.
Also, $C(m,1)=1$, since a single voter always has a top-ranked label.
These elementary cases explain the lower bounds in the following theorems.

\begin{theorem}\label{thm01}
For every $m\ge3$ and every odd $n\ge1$,
\[
C(m,n)>C(m,n+2).
\]
\end{theorem}

\begin{theorem}\label{thm02}
For every $m\ge2$ and every odd $n\ge3$,
\[
C(m,n)>C(m+1,n).
\]
\end{theorem}

A common idea in both proofs is to work with an even number of Bernoulli variables.
By adding or removing one variable, we express the corresponding majority probabilities in terms of
quantities associated with this even collection.
The key is the tie event, which occurs at the majority threshold only for an even number of variables.
These identities reduce the desired comparisons to integral positivity problems, which are then resolved by different arguments.
This also explains why our approach is restricted to odd $n$.

We collect the preliminary results in Section~\ref{sec2} and prove Theorems~\ref{thm02} and~\ref{thm01} in Sections~\ref{sec3} and~\ref{sec4}, respectively.

\section{Preliminaries}		\label{sec2}
\subsection{Integral representation of $C(m,n)$}
For  $n\ge 1$ and $\vec u=(u_1,\ldots,u_n)\in(0,1)^n$, let $X_1, \ldots, X_n$ be independent Bernoulli random variables with parameters $u_1, \ldots, u_n$, respectively.
Define
\[
W_n(\vec u) =\mathbb{P}\left(\sum_{i=1}^n X_i>\frac n2\right)
=\sum_{\substack{A\subset\{1,\ldots,n\}\\ \abs{A}>n/2}} \prod_{i\in A}u_i\prod_{i\notin A}(1-u_i).
\]
In particular, $W_n$ is a polynomial in $u_1,\ldots, u_n$.

The following integral representation appears in \cite[Theorem~5]{MB2002}; see also \cite{Sauermann2022, Pittel2025}.
We include a proof for completeness.

\begin{lemma}
For odd $n$, we have
\begin{equation}	\label{eq1902sat}
C(m,n)=m\int_{(0,1)^n}W_n(\vec u)^{m-1}\,d\vec u.
\end{equation}
\end{lemma}
\begin{proof}
For each of the $n$ permutations, assign a uniform score in $(0,1)$ to each of the $m$ labels, taking all scores to be mutually independent.
Order the labels by decreasing score in each permutation.
Since ties have probability zero, each ordering is a uniform random permutation of $\{1,\ldots,m\}$, and the resulting $n$ permutations are independent.

Fix a label $a$ and condition on its $n$ scores $\vec u=(u_1,\ldots,u_n)$.
For any rival $b$, in the $i$-th permutation, $a$ is placed before $b$ precisely when the score of $b$ is less than $u_i$.
Since that score is uniform on $(0,1)$, this occurs with probability $u_i$.
Thus, over the $n$ permutations, the comparisons of $a$ with $b$ are independent
Bernoulli trials with parameters $u_1,\ldots, u_n$.
Hence the conditional probability that $a$ defeats $b$ by a majority is $W_n(\vec u)$.

The scores of distinct rivals are independent, so conditional on $\vec u$, the events that $a$ defeats the different rivals are independent.
Therefore, the conditional probability that $a$ is a majority winner is
\[
W_n(\vec u)^{m-1}.
\]

Since the $n$ scores of $a$ are independent and uniform on $(0,1)$,
integrating over $\vec u\in(0,1)^n$ gives
\[
\mathbb P(a\text{ is a majority winner}) = \int_{(0,1)^n} W_n(\vec u)^{m-1}\,d\vec u.
\]
Finally, the $m$ labels are symmetric, and the events that they are majority winners are pairwise disjoint.
Summing over the $m$ labels yields \eqref{eq1902sat}.
\end{proof}

\subsection{Identities for majority and tie probabilities}
For the remainder of this section, let $N \ge 2$ be even and $\vec u=(u_1,\ldots,u_N)\in(0,1)^N$.
Let $X_1,\ldots,X_N$ be independent Bernoulli random variables satisfying $\mathbb P(X_i=1)=u_i$, and set
\[
\textstyle S=\sum_{i=1}^N X_i.
\]
Define
\[
w(\vec u)=\mathbb P(S>N/2),\qquad \tau(\vec u)=\mathbb P(S=N/2),\qquad \pi_i(\vec u)=\mathbb P(X_i=1\mid S=N/2).
\]
Here $w$ and $\tau$ are the probabilities of a strict majority and a tie, respectively, while $\pi_i$ is the conditional
probability that $X_i=1$ given a tie.
All these probabilities depend on $\vec u$, as indicated by the notation.

Since $0<u_i<1$, every configuration of the $X_i$'s has positive probability.
Consequently, $0<w<1$ and $\tau>0$.
Moreover, both $X_i=0$ and $X_i=1$ occur with positive probability on the event $S=N/2$, so $0<\pi_i<1$ for every $i=1,\ldots, N$.

The following identities express the majority probabilities obtained by adding or removing one variable in terms of the same majority and tie probabilities.
They provide the starting point for our monotonicity arguments.

\begin{lemma}
Suppose $N\ge2$ is even.
For every $\vec u\in(0,1)^N$ and $v\in(0,1)$, we have
\begin{align}
W_{N+1}(u_1,\ldots,u_N,v) &=w(\vec u)+\tau(\vec u)v,	\label{eq1921sat}\\
W_{N-1}(u_1,\ldots,\widehat{u_i},\ldots,u_N) &=w(\vec u)+\tau(\vec u)(1-\pi_i(\vec u)), \quad i=1,\ldots,N,	\label{eq1922sat}\\
\sum_{i=1}^N\pi_i(\vec u) &=N/2.	\label{eq1944sat}
\end{align}
Here, $\widehat{u_i}$ indicates that the $i$-th argument is omitted. 
Moreover,
\begin{equation}	\label{eq1947sat}
\tau(\vec 1-\vec u)=\tau(\vec u),\qquad w(\vec 1-\vec u)=1-w(\vec u)-\tau(\vec u),
\end{equation}
where $\vec 1-\vec u=(1-u_1,\ldots,1-u_N)$.
\end{lemma}
\begin{proof}
Adding an independent Bernoulli variable with parameter $v$ produces a strict majority if either $S>N/2$, or $S=N/2$ and the new variable equals $1$.
These disjoint cases give \eqref{eq1921sat}.

If $X_i$ is omitted, a strict majority among the remaining $N-1$ variables occurs precisely when $S-X_i\ge N/2$.
Since
\[
\{S-X_i\ge N/2\} =\{S>N/2\}\cup\{S=N/2,\ X_i=0\},
\]
where the union is disjoint, we obtain \eqref{eq1922sat}.

Next, since
\[
\sum_{i=1}^N\pi_i =\mathbb E[S\mid S=N/2]=N/2,
\]
we obtain \eqref{eq1944sat}.

Finally, replacing $\vec u$ by $\vec 1-\vec u$ replaces the distribution of $S$ by that of $N-S$, which yields \eqref{eq1947sat}.
\end{proof}
\subsection{A change of variables and its properties}

We use \emph{log-odds} coordinates, as in the theory of conditional Bernoulli distributions; see, for example, \cite{CL1997, Chen2000}.

Introduce coordinates $(x,\lambda_1,\ldots,\lambda_{N-1})\in\mathbb R^N$ on $(0,1)^N$ by setting
\begin{equation}	\label{eq2018sat}
u_i=\frac{e^{x+\lambda_i}}{1+e^{x+\lambda_i}}, \quad 1\le i<N, \qquad u_N=\frac{e^x}{1+e^x}.
\end{equation}
The inverse transformation is given by
\[
x=\log\frac{u_N}{1-u_N}, \qquad \lambda_i=\log\frac{u_i}{1-u_i}-x, \quad 1\le i<N.
\]
For $\vec\lambda=(\lambda_1,\ldots,\lambda_N)\in\mathbb R^N$, define
\[
H(\vec\lambda) =\sum_{\substack{A\subset\{1,\ldots,N\}\\ \abs{A}=N/2}} e^{\sum_{i\in A}\lambda_i}, \qquad J(\vec\lambda) =\frac{e^{\sum_{i=1}^N\lambda_i}}{H(\vec\lambda)^2}.
\]
We have defined $H$ and $J$ on all of $\mathbb R^N$ so that derivatives with respect to every $\lambda_i$, including $\lambda_N$, are available.

\begin{lemma}		\label{lem03}
Suppose $N \ge 2$ is even. In the coordinates \eqref{eq2018sat}, write $\vec\lambda=(\lambda_1,\ldots,\lambda_{N-1},0)$.
Then
\begin{equation}	\label{eq2033sat}
\begin{aligned}
\tau(x,\lambda_1,\ldots,\lambda_{N-1}) &=\frac{e^{Nx/2}H(\vec\lambda)}{\prod_{i=1}^N(1+e^{x+\lambda_i})},\\
\pi_i(x,\lambda_1,\ldots,\lambda_{N-1}) &=\frac{\partial_{\lambda_i}H(\vec\lambda)}{H(\vec\lambda)},
\qquad 1\le i\le N.
\end{aligned}
\end{equation}
In particular, $\pi_i$ is independent of $x$.
Moreover, the change of variables satisfies
\begin{equation}	\label{eq2034sat}
d\vec u=\tau^2J(\vec\lambda)\,dx\,d\lambda_1\cdots d\lambda_{N-1}.
\end{equation}

The second formula in \eqref{eq2033sat} defines an extension, still denoted by $\pi_i(\vec\lambda)$, to all of $\mathbb R^N$.
Both $J$ and $\pi_i$ are invariant under common translations
\[
\vec\lambda\mapsto\vec\lambda+c(1,\ldots,1).
\]
Furthermore, for every $\vec\lambda\in\mathbb R^N$,
\begin{equation}	\label{eq2035sat}
\partial_{\lambda_i}\log J=1-2\pi_i, \qquad  \partial_{\lambda_i}\pi_i=\pi_i(1-\pi_i), \qquad 1\le i\le N.
\end{equation}
The ordering of the $\pi_i$ agrees with that of the $\lambda_i$:
\begin{equation}	\label{eq2121sat}
\pi_i<\pi_j\iff\lambda_i<\lambda_j.
\end{equation}

Finally, for each fixed $(\lambda_1,\ldots,\lambda_{N-1})$, the map $x\mapsto w(x,\lambda_1,\ldots,\lambda_{N-1})$ is a smooth, strictly increasing bijection from $\mathbb R$ onto $(0,1)$.
Writing $s=\partial_x w$, we have
\begin{equation}	\label{eq2122sat}
\frac{s}{\tau}=\sum_{i=1}^N u_i(1-\pi_i), \qquad \partial_x\left(\frac{s}{\tau}\right)=\sum_{i=1}^N u_i(1-u_i)(1-\pi_i)>0.
\end{equation}
\end{lemma}
\begin{proof}
For each $A\subset\{1,\ldots,N\}$ with $\abs{A}=N/2$, independence gives
\[
\mathbb P\left(\{i:X_i=1\}=A\right) =\prod_{i\in A}u_i\prod_{i\notin A}(1-u_i)=\frac{e^{Nx/2+\sum_{i\in A}\lambda_i}} {\prod_{i=1}^N(1+e^{x+\lambda_i})},
\]
where $\lambda_N=0$.
Summing over all such subsets gives the formula for $\tau$ in \eqref{eq2033sat}.
Similarly,
\[
\partial_{\lambda_i}H =\sum_{\substack{A\subset\{1,\ldots,N\}\\ \abs{A}=N/2,\ i\in A}} e^{\sum_{j\in A}\lambda_j}.
\]
Restricting the preceding probability sum to subsets containing $i$ and dividing by $\tau$ therefore gives
\[
\pi_i=\frac{\partial_{\lambda_i}H}{H}.
\]
In particular, $\pi_i$ is independent of $x$.

Writing $h_i=u_i(1-u_i)$, the Jacobian matrix with respect to $(\lambda_1,\ldots,\lambda_{N-1},x)$ is upper triangular, with diagonal entries $h_1,\ldots,h_N$.
Thus, its determinant is
\[
\prod_{i=1}^N u_i(1-u_i) =\frac{e^{Nx+\sum_{i=1}^N\lambda_i}} {\prod_{i=1}^N(1+e^{x+\lambda_i})^2} =\tau^2J,
\]
which proves \eqref{eq2034sat}.

Since
\[
H(\vec\lambda+c(1,\ldots,1)) =e^{Nc/2}H(\vec\lambda),
\]
both $J$ and $\pi_i$ are invariant under common translations.
Differentiating the definition of $J$ gives
\[
\partial_{\lambda_i}\log J=1-2\pi_i.
\]
Moreover, since $\partial_{\lambda_i}^2H=\partial_{\lambda_i}H$, we obtain
\[
\partial_{\lambda_i}\pi_i = \partial_{\lambda_i} \left(\frac{\partial_{\lambda_i}H}{H}\right)=\frac{\partial_{\lambda_i}H}{H} -\left(\frac{\partial_{\lambda_i}H}{H}\right)^2 =\pi_i(1-\pi_i),
\]
proving \eqref{eq2035sat}.

For $i\ne j$, recall that
\[
\pi_i-\pi_j =\frac{\partial_{\lambda_i}H- \partial_{\lambda_j}H}{H}.
\]
In the difference of the two derivatives, the terms corresponding to subsets containing both $i$ and $j$
cancel.
Every remaining subset contains exactly one of $i$ and $j$.
Pairing the subsets that differ only by replacing $i$ with $j$, we obtain
\[
\pi_i-\pi_j =\frac{e^{\lambda_i}-e^{\lambda_j}}{H(\vec\lambda)} \sum_{\substack{ A\subset\{1,\ldots,N\}\setminus\{i,j\}\\ \abs{A}=N/2-1}} e^{\sum_{\ell\in A}\lambda_\ell}.
\]
The sum is strictly positive, including when $N=2$, in which case it equals $1$.
This proves \eqref{eq2121sat}.

Next, let $S_i=S-X_i$.
By conditioning on $X_i$ and using independence, we have
\[
w=\mathbb P(S_i + X_i > N/2)= u_i\mathbb P(S_i\ge N/2)  +(1-u_i)\mathbb P(S_i>N/2).
\]
Since the distribution of $S_i$ does not depend on $u_i$, differentiation gives
\[
\frac{\partial w}{\partial u_i} =\mathbb P(S_i\ge N/2) -\mathbb P(S_i>N/2) =\mathbb P(S_i=N/2).
\]
On the other hand, by the definition of $\pi_i$ and independence,
\[
\tau(1-\pi_i) =\mathbb P(S=N/2,\ X_i=0)=\mathbb P(S_i=N/2,\ X_i=0)=(1-u_i)\mathbb P(S_i=N/2).
\]
Combining these identities and using $\partial_xu_i=u_i(1-u_i)$, the chain rule gives
\[
s=\partial_x w =\sum_{i=1}^N u_i(1-u_i) \frac{\partial w}{\partial u_i} =\tau\sum_{i=1}^N u_i(1-\pi_i).
\]
Since $\pi_i$ is independent of $x$, differentiating once more yields
\[
\partial_x\left(\frac{s}{\tau}\right) =\sum_{i=1}^N u_i(1-u_i)(1-\pi_i)>0,
\]
which proves \eqref{eq2122sat}.

Finally, $s>0$, so $w$ is strictly increasing in $x$.
For each fixed $(\lambda_1,\ldots,\lambda_{N-1})$,
we have $u_i\to0$ as $x\to-\infty$ and $u_i\to1$
as $x\to+\infty$ for every $i$.
Since
\[
0 \le w=\mathbb P(S>N/2) \le\mathbb P(S\ge1) \le\sum_{i=1}^N u_i
\]
and
\[
1 \ge w=\mathbb P(S>N/2) \ge\mathbb P(S=N) =\prod_{i=1}^N u_i,
\]
it follows that $w$ tends to $0$ and $1$, respectively.
Thus, the map $x\mapsto w(x,\lambda_1,\ldots,\lambda_{N-1})$ is a smooth, strictly increasing bijection from $\mathbb R$ onto $(0,1)$.
\end{proof}

If a differentiable function $G$ on $\mathbb R^N$ is invariant under common translations, then
\[
G(\vec\lambda+c(1,\ldots,1))=G(\vec\lambda)
\]
for every $c\in\mathbb R$.
Differentiating with respect to $c$ at $c=0$ and using the chain rule, we obtain
\begin{equation}	\label{eq2213sat}
\sum_{i=1}^N\partial_{\lambda_i}G=0.
\end{equation}

\subsection{Integrability of $1/\tau$}
We will also need the following integrability estimate.
\begin{lemma}	\label{lem04}
Suppose $N\ge2$ is even. Then
\[
\int_{(0,1)^N}\frac{d\vec u}{\tau(\vec u)} \le 3^{N/2}<\infty.
\]
\end{lemma}

\begin{proof}
Group the Bernoulli variables into $N/2$ pairs $(X_{2j-1},X_{2j})$, $j=1,\ldots,N/2$.
Since
\[
\bigcap_{j=1}^{N/2}\{X_{2j-1}+X_{2j}=1\} \subset\{S=N/2\},
\]
independence gives
\[
\tau(\vec u) \ge \prod_{j=1}^{N/2} \left(u_{2j-1}(1-u_{2j}) +(1-u_{2j-1})u_{2j}\right).
\]
Consequently, Tonelli's theorem gives
\[
\int_{(0,1)^N}\frac{d\vec u}{\tau(\vec u)} \le \left( \int_0^1\int_0^1 \frac{da\,db}{a+b-2ab} \right)^{N/2}.
\]

We show that the double integral is at most $3$.
On $(0,1/2)^2$, we have $a+b-2ab\ge\max(a,b)$, and hence
\[
\int_0^{1/2}\int_0^{1/2} \frac{da\,db}{a+b-2ab} \le 2\int_0^{1/2}\int_0^b\frac{da\,db}{b} =1.
\]
By the change of variables $(a,b)\mapsto(1-a,1-b)$, the integral over $(1/2,1)^2$ is also at most $1$.
On each of the two remaining squares, $a+b-2ab\ge1/2$, so each contributes at most $1/2$.
Thus, the double integral is at most $1+1+1/2+1/2=3$, as required.
\end{proof}

\section{Proof of Theorem~\ref{thm02}}	\label{sec3}
Fix $m\ge2$ and an odd integer $n\ge3$, and set $N=n-1$.
We use the notation from Section~\ref{sec2} with this even $N \ge 2$.

Set $g(y)=y^m(1-y)$, so that $g'(y)=my^{m-1}-(m+1)y^m$.
Writing $v=u_n$ and $\vec u=(u_1,\ldots,u_N)$, we obtain from \eqref{eq1902sat} and \eqref{eq1921sat} that
\begin{align*}
C(m,n)-C(m+1,n) &=\int_{(0,1)^n}g'(W_n)\,du_1\cdots du_n\\
&=\int_{(0,1)^N}\int_0^1 g'(w+\tau v)\,dv\,d\vec u\\
&=\int_{(0,1)^N} \frac{g(w+\tau)-g(w)}{\tau}\,d\vec u.
\end{align*}
Since $0<w<w+\tau<1$, both terms in the last integrand are bounded in absolute value by $\norm{g}_{L^\infty(0,1)}/\tau$.
They are therefore separately integrable by Lemma~\ref{lem04}.

Apply the reflection $\vec u\mapsto\vec 1-\vec u$ to the term containing $g(w+\tau)$.
By \eqref{eq1947sat}, this transformation preserves Lebesgue measure and $\tau$, and
\[
w(\vec 1-\vec u)+\tau(\vec 1-\vec u) =1-w(\vec u).
\]
Thus,
\[
\int_{(0,1)^N}\frac{g(w+\tau)}{\tau}\,d\vec u =\int_{(0,1)^N}\frac{g(1-w)}{\tau}\,d\vec u.
\]
Consequently,
\[
C(m,n)-C(m+1,n) =\int_{(0,1)^N} \frac{g(1-w)-g(w)}{\tau}\,d\vec u.
\]

We now use the coordinates of \eqref{eq2018sat}, with $\vec\lambda=(\lambda_1,\ldots,\lambda_{N-1},0)$.
For each fixed $\vec\lambda$ and $y\in(0,1)$, let $x=x_{\vec\lambda}(y)$ denote the unique solution
of $w(x,\vec\lambda)=y$, whose existence follows from Lemma~\ref{lem03}.
Writing $s=\partial_x w$, define
\[
\rho_{\vec\lambda}(y) =\frac{\tau(x_{\vec\lambda}(y),\vec\lambda)} {s(x_{\vec\lambda}(y),\vec\lambda)}, \qquad 0<y<1.
\]
Since $s/\tau$ is strictly increasing in $x$ by \eqref{eq2122sat}, and $w$ is strictly increasing in $x$, the function $\rho_{\vec\lambda}$ is strictly decreasing in $y$.

By the Jacobian formula \eqref{eq2034sat} and the substitution $y=w(x,\vec\lambda)$, for which $dy=s\,dx$, we obtain
\[
\frac{d\vec u}{\tau} =\tau J(\vec\lambda)\,dx\, d\lambda_1\cdots d\lambda_{N-1}
=J(\vec\lambda)\rho_{\vec\lambda}(y)\,dy\, d\lambda_1\cdots d\lambda_{N-1}.
\]
It follows that
\[
C(m,n)-C(m+1,n) =\int_{\mathbb R^{N-1}}J(\vec\lambda) \int_0^1 [g(1-y)-g(y)] \rho_{\vec\lambda}(y)\,dy\, d\lambda_1\cdots d\lambda_{N-1}.
\]
The integral above is absolutely convergent.
Indeed, the preceding change of variables and Lemma~\ref{lem04} give
\[
\int_{\mathbb R^{N-1}}J(\vec\lambda) \int_0^1\rho_{\vec\lambda}(y)\,dy\, d\lambda_1\cdots d\lambda_{N-1} =\int_{(0,1)^N}\frac{d\vec u}{\tau}<\infty.
\]
Since $\abs{g(1-y)-g(y)} \le 2 \norm{g}_{L^\infty(0,1)}$, absolute convergence follows.

Splitting the inner integral at $y=1/2$ and substituting $y\mapsto1-y$ in the second half, we obtain
\begin{multline}		\label{eq2156sat}
C(m,n)-C(m+1,n)\\
=\int_{\mathbb R^{N-1}}J(\vec\lambda) \int_0^{1/2} \left(g(1-y)-g(y)\right) \left(\rho_{\vec\lambda}(y) -\rho_{\vec\lambda}(1-y)\right)\,dy\, d\lambda_1\cdots d\lambda_{N-1}.
\end{multline}
For $0<y<1/2$, we have
\[
g(1-y)-g(y) =y(1-y)\left((1-y)^{m-1}-y^{m-1}\right)>0.
\]
Also, since $\rho_{\vec\lambda}$ is strictly decreasing,
\[
\rho_{\vec\lambda}(y) -\rho_{\vec\lambda}(1-y)>0.
\]
As $J(\vec\lambda)>0$, the integral in \eqref{eq2156sat} is strictly positive.
Thus $C(m,n)>C(m+1,n)$.
\qed

\section{Proof of Theorem~\ref{thm01}}	\label{sec4}
Fix $m\ge3$ and an odd integer $n\ge1$, and set $N=n+1$.
Then $N \ge 2$ is even, and throughout this section, we use the functions $w$, $\tau$, $\pi_i$, and $J$ from Section~\ref{sec2} with this value of $N$.
We will prove that $C(m,n)>C(m,n+2)$, as asserted in Theorem~\ref{thm01}.

\subsection{Reduction to $\Delta_k$}
By \eqref{eq1902sat} and \eqref{eq1922sat}, for each $i=1,\ldots,N$, we have
\[
\frac{C(m,n)}{m} =\int_{(0,1)^N} \left(w+\tau(1-\pi_i)\right)^{m-1}\,d\vec u.
\]
Indeed, by \eqref{eq1922sat}, the integrand equals
\[
W_{N-1}(u_1,\ldots,\widehat{u_i},\ldots,u_N)^{m-1},
\]
which is independent of $u_i$.
Thus, integration with respect to $u_i$ contributes a factor of $1$.
Averaging over $i=1,\ldots,N$, we obtain
\[
\frac{C(m,n)}{m} =\int_{(0,1)^N} \frac1N\sum_{i=1}^N \left(w+\tau(1-\pi_i)\right)^{m-1}\,d\vec u.
\]
Similarly, \eqref{eq1902sat} and \eqref{eq1921sat} give
\[
\frac{C(m,n+2)}{m} =\int_{(0,1)^N}\int_0^1 (w+\tau v)^{m-1}\,dv\,d\vec u.
\]

Expanding both expressions by the binomial theorem, the coefficient of $w^{m-1-k}\tau^k$ in their difference is
\[
\binom{m-1}{k} \left(\frac1N\sum_{i=1}^N(1-\pi_i)^k -\int_0^1v^k\,dv \right).
\]
For $k=0$, this coefficient vanishes.
It also vanishes for $k=1$, since
\[
\frac1N\sum_{i=1}^N(1-\pi_i) =\frac12=\int_0^1v\,dv
\]
by \eqref{eq1944sat}.
Since $\int_0^1v^k\,dv=1/(k+1)$, we obtain
\begin{equation}	\label{eq2202sat}
\frac{C(m,n)-C(m,n+2)}{m} =\sum_{k=2}^{m-1}\binom{m-1}{k} \,\Delta_k,
\end{equation}
where
\begin{equation}	\label{eq2205sat}
\Delta_k =\int_{(0,1)^N}\tau^k w^{m-1-k} \left\{ \frac1N\sum_{i=1}^N(1-\pi_i)^k -\frac1{k+1} \right\}\,d\vec u.
\end{equation}
We will prove that $\Delta_k>0$ for every $2\le k\le m-1$.

\subsection{Monotonicity along level sets of $w$}
We study how $\tau$ and $s=\partial_x w$ change when the parameters vary while $w$ remains fixed.
We first recall a well-known log-concavity property of sums of independent Bernoulli random variables.
\begin{lemma}	\label{lem05}
Let $T$ be a sum of finitely many independent Bernoulli random variables, and set $p_j=\mathbb P(T=j)$, with $p_j=0$ outside the support of $T$. Then
\[
p_j^2\ge p_{j-1}p_{j+1} \qquad\text{for every integer }j.
\]
\end{lemma}
\begin{proof}
This follows from the log-concavity of elementary symmetric functions; see \cite[Section 2.22, Theorem 53]{HLP1952}.
\end{proof}

For the next two lemmas, we extend the parametrization by allowing $\lambda_N$ to vary.
For $\vec\lambda\in\mathbb R^N$, set
\[
u_i=\frac{e^{x+\lambda_i}}{1+e^{x+\lambda_i}}, \qquad 1\le i\le N,
\]
and retain the notation $w(x,\vec\lambda)$ and $\tau(x,\vec\lambda)$ for the corresponding probabilities.
The formulas \eqref{eq2033sat} and \eqref{eq2122sat} remain valid, and for each fixed $\vec\lambda$, the map $x\mapsto w(x,\vec\lambda)$ remains a smooth, strictly increasing bijection from $\mathbb R$ onto $(0,1)$.
Moreover,
\begin{equation}	\label{eq2302sun}
\tau(x,\vec\lambda)=O(e^{-N\abs{x}/2})
\qquad (\abs{x}\to\infty).
\end{equation}
Since $x\mapsto w(x,\vec\lambda)$ is a smooth, strictly increasing bijection from $\mathbb R$ onto $(0,1)$ for every fixed $\vec\lambda\in\mathbb R^N$, given $\vec\lambda,\vec v\in\mathbb R^N$ and $y\in(0,1)$, there exists a unique $x=x(t)$ for each $t\in\mathbb R$ such that
\[
w(x(t),\vec\lambda+t\vec v)=y.
\]

Moreover, since $\partial_x w>0$, the implicit function theorem shows that $x(t)$ is smooth in a neighborhood of every $t\in\mathbb R$, and hence on all of $\mathbb R$.

The following lemma establishes a monotonicity property of the tie probability and the derivative of the majority probability along level sets of $w$.
\begin{lemma}	\label{lem06}
Suppose $N\ge2$ is even, and let $\vec\lambda,\vec v\in\mathbb R^N$ satisfy
\begin{equation}	\label{eq2208sat}
(v_i-v_j)(\lambda_i-\lambda_j)\ge0 \quad\text{for all }\; i, j.
\end{equation}
For any fixed $y\in(0,1)$, let $x=x(t)$ be defined as above.
Writing $s=\partial_x w$, and using a dot to denote differentiation along this curve at $t=0$, we have
\[
\dot\tau\ge0,\qquad \dot s \le 0.
\]
Moreover, $\dot s<0$ if the inequality in \eqref{eq2208sat} is strict for at least one pair.
\end{lemma}

\begin{proof}
Set
\[
h_i=u_i(1-u_i),\qquad \gamma_i=\mathbb P(S-X_i=N/2).
\]
Since $\partial_{u_i}w=\gamma_i$ and $\partial_x u_i=h_i$, the chain rule gives
\[
s=\partial_xw=\sum_{i=1}^N h_i\gamma_i>0.
\]

Along the curve $t\mapsto(x(t),\vec\lambda+t\vec v)$, we have
\[
u_i(t)=\frac{e^{x(t)+\lambda_i+tv_i}}{1+e^{x(t)+\lambda_i+t v_i}},
\]
and hence
\[
\dot u_i=h_i(\dot x+v_i).
\]
Differentiating $w(x(t),\vec\lambda+t\vec v)=y$ gives
\[
0=\sum_{i=1}^N\gamma_i h_i(\dot x+v_i).
\]
Consequently,
\[
\dot x=-\eta,\qquad \dot u_i=h_i(v_i-\eta),\qquad \eta=\frac1s \sum_{j=1}^N h_j\gamma_j v_j.
\]

For any smooth function $f=f(\vec u)$, the chain rule gives
\[
\dot f =\sum_{i=1}^N \frac{\partial f}{\partial u_i}\dot u_i =\sum_{i=1}^N \frac{\partial f}{\partial u_i} h_i(v_i-\eta)\\
=\frac1s \sum_{i=1}^N \sum_{j=1}^N h_i h_j\frac{\partial f}{\partial u_i} \gamma_j(v_i-v_j),
\]
where the last equality follows from $s=\sum_{j=1}^N h_j\gamma_j$ and the definition of $\eta$.
Pairing the terms corresponding to $(i,j)$ and $(j,i)$, we obtain
\begin{equation}	\label{eq2214sat}
\dot f =\frac1s\sum_{i<j}h_ih_j(v_i-v_j) \left(\frac{\partial f}{\partial u_i} \gamma_j-  \frac{\partial f}{\partial u_j} \gamma_i \right).
\end{equation}

We first apply this identity to $f=\tau$.
Fix $i\ne j$, let $T=S-X_i-X_j$, and set
\[
a=\mathbb P(T=N/2-2),\qquad b=\mathbb P(T=N/2-1),\qquad c=\mathbb P(T=N/2).
\]
Since $S-X_i=T+X_j$, conditioning on $X_j$ and using independence gives
\[
\gamma_i=(1-u_j)c+u_jb=c+u_j(b-c).
\]

On the other hand, conditioning on $X_i$ gives
\[
\tau =u_i\mathbb P(S-X_i=N/2-1) +(1-u_i)\mathbb P(S-X_i=N/2).
\]
Since the distribution of $S-X_i$ does not depend on $u_i$, differentiation yields
\[
\frac{\partial\tau}{\partial u_i} =\mathbb P(S-X_i=N/2-1) -\mathbb P(S-X_i=N/2).
\]
Conditioning on $X_j$, we have
\[
\begin{aligned}
\mathbb P(S-X_i=N/2-1)&=(1-u_j)b+u_ja,\\
\mathbb P(S-X_i=N/2)&=(1-u_j)c+u_jb.
\end{aligned}
\]
Consequently,
\[
\frac{\partial\tau}{\partial u_i} =b-c+u_j(a-2b+c).
\]
Interchanging $i$ and $j$ and simplifying, we obtain
\[
\begin{aligned}
\frac{\partial\tau}{\partial u_i}\gamma_j -\frac{\partial\tau}{\partial u_j}\gamma_i &= (u_i-u_j) \left((b-c)^2-c(a-2b+c)\right)\\
&=(u_i-u_j)(b^2-ac).
\end{aligned}
\]
By Lemma~\ref{lem05}, $b^2\ge ac$.
Moreover, since $u_i$ is strictly increasing in $\lambda_i$ for fixed $x$, the difference $u_i-u_j$ has the same sign as $\lambda_i-\lambda_j$.
Thus, condition \eqref{eq2208sat} implies
\[
(v_i-v_j)(u_i-u_j)\ge 0.
\]
Since $h_i,h_j,s>0$, every summand in \eqref{eq2214sat} is nonnegative, and hence $\dot\tau\ge0$.
The argument also applies when $N=2$, in which case $(a,b,c)=(0,1,0)$.

We next apply \eqref{eq2214sat} to $f=s$.
For each $i$, conditioning on $X_i$ gives
\[
\begin{aligned}
\mathbb E\left[(X_i-u_i)\mathbbm1_{\{S>N/2\}}\right]
&=u_i(1-u_i) [\mathbb P(S-X_i\ge N/2) -\mathbb P(S-X_i>N/2)]\\
&=h_i\gamma_i.
\end{aligned}
\]
Summing over $i$, we obtain
\[
s=\mathbb E\left[ S\mathbbm1_{\{S>N/2\}} \right] -\left(\textstyle\sum_{\ell=1}^N u_\ell \right)w.
\]
Since the distribution of $S-X_i$ does not depend on $u_i$, conditioning on $X_i$ gives
\[
\mathbb E\left[S\mathbbm1_{\{S>N/2\}}\right] =(1-u_i)\mathbb E\left[(S-X_i)\mathbbm1_{\{S-X_i>N/2\}} \right] +u_i\mathbb E\left[(S-X_i+1)\mathbbm1_{\{S-X_i\ge N/2\}} \right].
\]
Differentiating with respect to $u_i$, we obtain
\[
\frac{\partial}{\partial u_i} \mathbb E\left[S\mathbbm1_{\{S>N/2\}}\right] =\mathbb P(S-X_i>N/2) +(N/2+1)\gamma_i.
\]
On the other hand, conditioning on $X_i$ gives
\[
\begin{aligned}
w&=(1-u_i)\mathbb P(S-X_i>N/2)+u_i\mathbb P(S-X_i\ge N/2)\\
&=\mathbb P(S-X_i>N/2)+u_i\gamma_i.
\end{aligned}
\]
Using $\partial_{u_i}w=\gamma_i$, we obtain
\[
\begin{aligned}
\frac{\partial s}{\partial u_i} &=\mathbb P(S-X_i>N/2) +(N/2+1)\gamma_i - w -\left(\textstyle\sum_{\ell=1}^N u_\ell \right)\gamma_i\\
&=\left( N/2+1-\textstyle\sum_{\ell=1}^N u_\ell -u_i\right)\gamma_i.
\end{aligned}
\]
It follows that, for $i\ne j$,
\[
\frac{\partial s}{\partial u_i}\gamma_j -\frac{\partial s}{\partial u_j}\gamma_i =-(u_i-u_j)\gamma_i\gamma_j.
\]
Substituting this identity into \eqref{eq2214sat} gives
\[
\dot s =-\frac1s\sum_{i<j} h_i h_j\gamma_i\gamma_j (v_i-v_j)(u_i-u_j) \le 0.
\]
Finally, since $h_i>0$ and $\gamma_i>0$ for every $i$, and since $u_i-u_j$ has the same sign as $\lambda_i-\lambda_j$, the inequality is strict if \eqref{eq2208sat} is strict for at least one pair.
Hence $\dot s<0$.
\end{proof}

\subsection{Monotonicity of $I_k$}
For $2\le k\le m-1$, define
\[
I_k(\vec\lambda) =\int_{\mathbb R} \tau(x,\vec\lambda)^{k+2} w(x,\vec\lambda)^{m-1-k}\,dx.
\]

We next use Lemma~\ref{lem06} to establish a monotonicity property of $I_k$ that will play a key role in proving Lemma~\ref{lem08}.
\begin{lemma}	\label{lem07}
Suppose $N\ge2$ is even and $2\le k\le m-1$.
Then $I_k\in C^1(\mathbb R^N)$, is invariant under common translations, and satisfies
\begin{equation}	\label{eq2218sat}
\abs{\partial_{\lambda_i} I_k} \le(m+1)I_k,\qquad 1\le i\le N.
\end{equation}
Moreover, if $\vec\lambda,\vec v\in\mathbb R^N$ satisfy \eqref{eq2208sat}, then
\[
\sum_{i=1}^N v_i\partial_{\lambda_i}I_k(\vec\lambda) \ge 0.
\]
The inequality is strict if $(v_i-v_j)(\lambda_i-\lambda_j)>0$ for at least one pair.
\end{lemma}

\begin{proof}
The integral defining $I_k$ converges by \eqref{eq2302sun}, since $0<w<1$.
A common translation of $\vec\lambda$ can be absorbed by a translation of $x$, so $I_k$ is invariant under common translations.

We first justify differentiation under the integral sign.
For any fixed event $E$ determined by the Bernoulli variables, let $p(\vec u)=\mathbb P(E)$.
Conditioning on $X_i$ gives
\[
p=u_i\mathbb P(E\mid X_i=1) +(1-u_i)\mathbb P(E\mid X_i=0).
\]
Since the conditional probabilities do not depend on $u_i$, and $\partial_{\lambda_i}u_i=u_i(1-u_i)$ with $x$ fixed, differentiation yields
\[
\partial_{\lambda_i}p =u_i(1-u_i)\left[\mathbb P(E\mid X_i=1)-\mathbb P(E\mid X_i=0)\right].
\]
On the other hand, conditioning on $X_i$ gives
\[
\mathbb E[(X_i-u_i)\mathbbm{1}_E] =u_i(1-u_i)\mathbb P(E\mid X_i=1) -u_i(1-u_i)\mathbb P(E\mid X_i=0).
\]
Consequently,
\[
\partial_{\lambda_i}p =\mathbb E[(X_i-u_i)\mathbbm{1}_E].
\]
Since $\abs{X_i-u_i} \le 1$, it follows that
\[
\abs{\partial_{\lambda_i}p} \le p.
\]
Applying this to $w$ and $\tau$ and using the product rule, we obtain
\begin{equation}	\label{eq1504mon}
\Abs{\partial_{\lambda_i} \left(\tau^{k+2}w^{m-1-k}\right)} \le(m+1)\tau^{k+2}w^{m-1-k}.
\end{equation}
Since $\tau, w>0$, this implies
\[
\Abs{\partial_{\lambda_i}\log\left(\tau^{k+2}w^{m-1-k}\right)} \le m+1.
\]
Integrating along coordinate directions and exponentiating, we obtain
\[
\tau(x,\vec\lambda')^{k+2} w(x,\vec\lambda')^{m-1-k} \le e^{(m+1)\sum_{\ell=1}^N \abs{\lambda_\ell'-\lambda_\ell}}\, \tau(x,\vec\lambda)^{k+2} w(x,\vec\lambda)^{m-1-k}.
\]
Combining this estimate with \eqref{eq1504mon}, we see that, in a neighborhood of any fixed $\vec\lambda$, the integrand and its first derivatives are bounded by a common integrable function
of $x$.
Differentiation under the integral sign is therefore justified.
This proves that $I_k\in C^1(\mathbb R^N)$ and establishes \eqref{eq2218sat}.

Now fix $\vec\lambda$ and $\vec v$ satisfying \eqref{eq2208sat}.
For each $t\in\mathbb R$ and $y\in(0,1)$, let $x=x(t,y)$ be the unique smooth solution of
\[
w(x(t,y),\vec\lambda+t\vec v)=y.
\]
Define
\[
F(t,y) =\frac{\tau(x(t,y),\vec\lambda+t\vec v)^{k+2}}{s(x(t,y),\vec\lambda+t\vec v)}.
\]
Since $s=\partial_x w>0$, the substitution $y=w(x,\vec\lambda+t\vec v)$ gives $dy=s\,dx$, and hence
\[
I_k(\vec\lambda+t\vec v) =\int_0^1 y^{m-1-k}F(t,y)\,dy.
\]

For $t\ge0$, the ordering condition \eqref{eq2208sat} remains valid, since
\[
(v_i-v_j) \left((\lambda_i+tv_i)-(\lambda_j+tv_j)\right) = (v_i-v_j)(\lambda_i-\lambda_j)+t(v_i-v_j)^2 \ge 0.
\]
We may therefore apply Lemma~\ref{lem06} at each $t \ge 0$, taking $\vec\lambda+t\vec v$ as the initial parameter vector.
This gives $\dot\tau\ge 0$ and $\dot s\le 0$ along the curve $t \mapsto (x(t,y),\vec\lambda+t\vec v)$.
Consequently, since $\tau,s>0$,
\[
\partial_t F =\frac{(k+2)\tau^{k+1}\dot\tau}{s} -\frac{\tau^{k+2}\dot s}{s^2} \ge 0.
\]
Thus $F(t,y)\ge F(0,y)$ for every $t \ge 0$ and $y \in (0,1)$.
Integrating against $y^{m-1-k}$ gives
\[
I_k(\vec\lambda+t\vec v) \ge I_k(\vec\lambda),\qquad t \ge 0.
\]
Since $I_k\in C^1(\mathbb R^N)$, taking the right derivative at $t=0$ yields
\[
\sum_{i=1}^N v_i\partial_{\lambda_i}I_k(\vec\lambda)\ge0.
\]

Finally, suppose that \eqref{eq2208sat} is strict for at least one pair.
By Lemma~\ref{lem06}, we have $\dot s<0$ at $t=0$ for every $y\in(0,1)$, and hence
\[
\partial_t F(0,y)>0,\qquad 0<y<1.
\]
Since $F(t,y)\ge F(0,y)$ for $t \ge 0$, we may restrict the integral to $[1/4,3/4]$ to obtain
\[
I_k(\vec\lambda+t\vec v)-I_k(\vec\lambda) \ge \int_{1/4}^{3/4} y^{m-1-k}[F(t,y)-F(0,y)]\,dy.
\]
The function $F$ is smooth in a neighborhood of $\{0\}\times[1/4,3/4]$.
Dividing by $t>0$ and letting $t\downarrow0$, we obtain
\[
\sum_{i=1}^N v_i\partial_{\lambda_i}I_k(\vec\lambda) \ge \int_{1/4}^{3/4} y^{m-1-k}\partial_tF(0,y)\,dy>0.
\]
This completes the proof.
\end{proof}

\subsection{Positivity of $\Delta_k$ by integration by parts}
The key step in proving Theorem~\ref{thm01} is the following strict positivity result.

\begin{lemma}	\label{lem08}
Let $N \ge 2$ be an even integer and $m \ge 3$ an integer.
Then $\Delta_k$ defined in \eqref{eq2205sat} is strictly positive for every integer $2\le k\le m-1$.
\end{lemma}

\begin{proof}
We now return to the slice $\lambda_N=0$, identified with $\mathbb R^{N-1}$.
Recall from \eqref{eq2034sat} that
\[
d\vec u=\tau^2J\,dx\,d\lambda_1\cdots d\lambda_{N-1}.
\]
Since $J$ and $\pi_i$ are independent of $x$, the definition of $I_k$ gives
\[
N \Delta_k =\int_{\mathbb R^{N-1}}J I_k \sum_{i=1}^N \left((1-\pi_i)^k-\frac1{k+1}\right)\,d\lambda_1\cdots d\lambda_{N-1}.
\]
Moreover, the same change of variables gives
\begin{equation}	\label{eq2216sat}
\int_{\mathbb R^{N-1}}J I_k\, d\lambda_1\cdots d\lambda_{N-1} = \int_{(0,1)^N} \tau^k w^{m-1-k}\,d\vec u \le 1.
\end{equation}

To obtain the divergence identity \eqref{eq2221sat},
we seek a bounded function $b$ satisfying
\begin{equation}	\label{eq2157tue}
\frac{d}{dz}\left[z(1-z)b(z)\right] =(1-z)^k-\frac1{k+1}.
\end{equation}
This is the Stein equation for the uniform distribution on $(0,1)$ with test function $h(z)=(1-z)^k$; see \cite[(2.3)]{GR2013}.
Integrating from $0$ to $z$ gives
\[
z(1-z)b(z) =\frac{(1-z)-(1-z)^{k+1}}{k+1}.
\]
Thus, we may take
\[
b(z)=\frac1{k+1}\sum_{\ell=0}^{k-1}(1-z)^\ell, \qquad 0\le z \le 1.
\]
Moreover, since $k\ge2$,
\begin{equation}	\label{eq2220sat}
b'(z) =-\frac1{k+1}\sum_{\ell=1}^{k-1} \ell(1-z)^{\ell-1} \le -\frac1{k+1}.
\end{equation}
Using \eqref{eq2035sat} and \eqref{eq2157tue}, we obtain
\begin{equation}	\label{eq0219sat}
\partial_{\lambda_i}\left[Jb(\pi_i)\right]  =J\left((1-\pi_i)^k-\frac1{k+1}\right).
\end{equation}

For $i<N$, set
\[
\beta_i=b(\pi_i)-b(\pi_N).
\]
Since $Jb(\pi_N)$ is invariant under common translations, \eqref{eq2213sat} gives
\[
\sum_{i<N}\partial_{\lambda_i} [Jb(\pi_N)] =-\partial_{\lambda_N}[Jb(\pi_N)].
\]
Consequently, using \eqref{eq0219sat}, we obtain
\begin{equation}	\label{eq2221sat}
\begin{aligned}
\sum_{i<N}\partial_{\lambda_i}(J\beta_i) &=\sum_{i<N}\partial_{\lambda_i}[Jb(\pi_i)] +\partial_{\lambda_N}[Jb(\pi_N)]\\
&=\sum_{i=1}^N\partial_{\lambda_i}[Jb(\pi_i)] =J\sum_{i=1}^N \left((1-\pi_i)^k-\frac1{k+1}\right).
\end{aligned}
\end{equation}
Here, all derivatives are taken in $\mathbb R^N$ before restricting to $\lambda_N=0$.

Combining the integral representation of $\Delta_k$ with \eqref{eq2221sat}, we have
\[
N\Delta_k =\int_{\mathbb R^{N-1}} I_k \sum_{i<N}\partial_{\lambda_i}(J\beta_i)\, d\lambda_1\cdots d\lambda_{N-1}.
\]
Formally integrating by parts yields
\begin{equation}	\label{eq2206sat}
N\Delta_k =-\int_{\mathbb R^{N-1}} J \sum_{i<N}\beta_i\partial_{\lambda_i}I_k\, d\lambda_1\cdots d\lambda_{N-1}.
\end{equation}
We now justify this integration by parts using a cutoff argument.

Choose $\chi\in C_c^\infty(\mathbb R^{N-1})$ such that $0\le\chi\le1$ and $\chi=1$ near the origin, and set
\[
\chi_\varepsilon(\lambda_1,\ldots,\lambda_{N-1}) =\chi(\varepsilon\lambda_1,\ldots,\varepsilon\lambda_{N-1}).
\]
Multiplying \eqref{eq2221sat} by $\chi_\varepsilon I_k$ and integrating by parts gives
\begin{multline*}
\int_{\mathbb R^{N-1}}\chi_\varepsilon JI_k \sum_{i=1}^N\left((1-\pi_i)^k-\frac1{k+1}\right) \,d\lambda_1\cdots d\lambda_{N-1}\\
=-\int_{\mathbb R^{N-1}}\chi_\varepsilon J \sum_{i<N}\beta_i\partial_{\lambda_i}I_k \,d\lambda_1\cdots d\lambda_{N-1}
-\int_{\mathbb R^{N-1}}JI_k \sum_{i<N}\beta_i\partial_{\lambda_i}\chi_\varepsilon \,d\lambda_1\cdots d\lambda_{N-1}.
\end{multline*}
Since $\abs{\beta_i} \le 1$ and $\abs{\partial_{\lambda_i}I_k} \le (m+1)I_k$ by Lemma~\ref{lem07}, the integrands on the left
and in the first integral on the right are bounded in absolute value by $NJI_k$ and $(N-1)(m+1)JI_k$, respectively.
These functions are integrable by \eqref{eq2216sat}.
Since $\chi_\varepsilon\to1$ pointwise, dominated convergence applies to both integrals.

For the last integral, we use
\[
\abs{\partial_{\lambda_i}\chi_\varepsilon} \le \varepsilon \norm{\nabla\chi}_\infty.
\]
Its absolute value is therefore bounded by
\[
(N-1)\varepsilon\norm{\nabla\chi}_\infty \int_{\mathbb R^{N-1}}JI_k\, d\lambda_1\cdots d\lambda_{N-1},
\]
which tends to zero by \eqref{eq2216sat}.
This proves \eqref{eq2206sat}.

We now verify that
\[
-\sum_{i<N}\beta_i\partial_{\lambda_i}I_k>0 \quad\text{whenever }\;\vec\lambda\ne\vec 0.
\]
For each $\vec\lambda$, define
\[
v_i=-b(\pi_i),\qquad 1\le i\le N.
\]
Since $I_k$ is invariant under common translations, \eqref{eq2213sat} gives $\sum_{i=1}^N\partial_{\lambda_i}I_k=0$.
Hence
\[
\begin{aligned}
-\sum_{i<N}\beta_i\partial_{\lambda_i}I_k &=-\sum_{i<N}b(\pi_i)\partial_{\lambda_i}I_k -b(\pi_N)\partial_{\lambda_N}I_k\\
&=-\sum_{i=1}^N b(\pi_i)\partial_{\lambda_i}I_k =\sum_{i=1}^N v_i\partial_{\lambda_i}I_k.
\end{aligned}
\]

By \eqref{eq2121sat}, the ordering of the $\pi_i$ agrees with that of the $\lambda_i$.
Since $b$ is strictly decreasing by \eqref{eq2220sat}, the ordering of the $v_i=-b(\pi_i)$ also agrees
with that of the $\lambda_i$.
Consequently,
\[
(v_i-v_j)(\lambda_i-\lambda_j)\ge0 \quad\text{for all }\; i,j,
\]
with strict inequality whenever $\lambda_i \ne \lambda_j$.

On the slice $\lambda_N=0$, if $\vec\lambda\ne\vec 0$, then $\lambda_i\ne\lambda_N$ for some $i<N$.
Applying Lemma~\ref{lem07} at each such point, with $\vec v$ held fixed, yields
\[
-\sum_{i<N}\beta_i\partial_{\lambda_i}I_k =\sum_{i=1}^N v_i\partial_{\lambda_i}I_k>0.
\]
At $\vec\lambda=\vec 0$, all $\pi_i$ are equal, so $\beta_i=0$ for every $i<N$.

Since $J>0$ and the origin has measure zero in $\mathbb R^{N-1}$, \eqref{eq2206sat} yields $\Delta_k>0$.
\end{proof}

\begin{proof}[Proof of Theorem~\ref{thm01}]
Since $N=n+1$ is even, Lemma~\ref{lem08} shows that every term in the sum \eqref{eq2202sat} is strictly positive.
Moreover, the sum is nonempty since $m \ge 3$.
Consequently,
\[
C(m,n)>C(m,n+2),
\]
as desired.
\end{proof}

\section*{Use of AI tools}
Large language models (LLMs) were used to explore proof ideas and assist with computational checks, literature searches, and language polishing.

The authors wrote the manuscript, verified the proofs, and take full responsibility for the content of this paper.


\begin{thebibliography}{99}

\bibitem{Black1948}
Black, Duncan.
\textit{On the rationale of group decision-making}.
Journal of Political Economy \textbf{56} (1948), no.~1, 23--34.
http://www.jstor.org/stable/1825026.

\bibitem{CDNPR2025}
Caizergues, Emma; Durand, Fran\c{c}ois; Noy, Marc; de Panafieu, \'Elie; Ravelomanana, Vlady.
\textit{Probability of a Condorcet winner for large electorates: An analytic combinatorics approach}.
arXiv:2505.06028 [cs.GT]
https://doi.org/10.48550/arXiv.2505.06028

\bibitem{Chen2000}
Chen, Sean X.
\textit{General properties and estimation of conditional Bernoulli models}.
J. Multivariate Anal. \textbf{74} (2000), no. 1, 69--87.
https://doi.org/10.1006/jmva.1999.1872

\bibitem{CL1997}
Chen, Sean X.; Liu, Jun S.
\textit{Statistical applications of the Poisson-Binomial and conditional Bernoulli distributions}.
Statist. Sinica \textbf{7} (1997), no. 4, 875--892.
http://www.jstor.org/stable/24306160

\bibitem{DP1970}
DeMeyer, Frank; Plott, Charles R.
\textit{The probability of a cyclical majority}.
Econometrica \textbf{38} (1970), no. 2, 345--354.
https://doi.org/10.2307/1913015

\bibitem{Fishburn1973}
Fishburn, Peter C.
\textit{A proof of May's theorem $p(m,4)=2p(m,3)$}.
Behavioral Science \textbf{18} (1973), no. 3, 212.
https://doi.org/10.1002/bs.3830180309

\bibitem{FGM1979}
Fishburn, Peter C.; Gehrlein, William V.; Maskin, Eric.
\textit{Condorcet proportions and Kelly's conjectures}.
Discrete Appl. Math. \textbf{1} (1979), no. 4, 229--252.
https://doi.org/10.1016/0166-218X(79)90001-5

\bibitem{Gehrlein2006}
Gehrlein, William V.
\textit{Condorcet's paradox}.
Theory and Decision Library C, \textbf{40}. Springer, Berlin--Heidelberg, 2006.
https://doi.org/10.1007/3-540-33799-7

\bibitem{GF1976}
Gehrlein, William V.; Fishburn, Peter C.
\textit{The probability of the paradox of voting: A computable solution}.
Journal of Economic Theory \textbf{13} (1976), no. 1, 14--25.
https://doi.org/10.1016/0022-0531(76)90063-6

\bibitem{GR2013}
Goldstein, Larry; Reinert, Gesine.
\textit{Stein's method for the beta distribution and the P\'olya-Eggenberger urn}.
J. Appl. Probab. \textbf{50} (2013), no. 4, 1187--1205.
https://doi.org/10.1239/jap/1389370107

\bibitem{HLP1952}
Hardy, G. H.; Littlewood, J. E.; P\'olya, G.
\textit{Inequalities}.
Reprint of the 1952 edition.
Cambridge Math. Lib.
Cambridge University Press, Cambridge, 1988. xii+324 pp.

\bibitem{Kelly1974}
Kelly, Jerry S.
\textit{Voting anomalies, the number of voters, and the number of alternatives}.
Econometrica \textbf{42} (1974), no. 2, 239--251.
https://doi.org/10.2307/1911974

\bibitem{Kelly1986}
Kelly, Jerry S.
\textit{Condorcet winner proportions}.
Social Choice and Welfare \textbf{3} (1986), no. 4, 311--314.
http://www.jstor.org/stable/41105845

\bibitem{KR2005}
Krishnamoorthy, M. S.; Raghavachari, M.
\textit{Condorcet winner probabilities - A statistical perspective}.
arXiv:math/0511140 [math.ST]
https://doi.org/10.48550/arXiv.math/0511140

\bibitem{MB2002}
Maassen, Hans; Bezembinder, Thom.
\textit{Generating random weak orders and the probability of a Condorcet winner}.
Social Choice and Welfare \textbf{19} (2002), no. 3, 517--532.
https://doi.org/10.1007/s003550100129

\bibitem{May1971}
May, Robert M.
\textit{Some mathematical remarks on the paradox of voting}.
Behavioral Science \textbf{16} (1971), no. 2, 143--151.
https://doi.org/10.1002/bs.3830160204

\bibitem{Pittel2025}
Pittel, Boris.
\textit{On likelihood of a Condorcet winner for uniformly random and independent voter preferences}.
arXiv:2506.20613 [math.CO]
https://doi.org/10.48550/arXiv.2506.20613

\bibitem{Sauermann2022}
Sauermann, Lisa.
\textit{On the probability of a Condorcet winner among a large number of alternatives}.
arXiv:2203.13713 [econ.TH]
https://doi.org/10.48550/arXiv.2203.13713

\end{thebibliography}
\end{document}